\documentclass[12pt]{article}

\usepackage{amssymb}
\usepackage{latexsym, amsmath, amscd,amsthm}
\usepackage{amsmath}
\usepackage{amssymb}
\usepackage{multirow}
\usepackage{amsthm}
\def\ord{\mathop{\operatorfont ord}\nolimits}

\newtheorem{theorem}{Theorem}[section]
\newtheorem{corollary}[theorem]{Corollary}
\newtheorem{proposition}[theorem]{Proposition}
\newtheorem{lemma}[theorem]{Lemma}

\theoremstyle{definition}

\newtheorem{example}[theorem]{Example}
\newtheorem{remark}[theorem]{Remark}

\newcommand{\Cn}{{\cal C}_n}
\newcommand{\Xn}{{\cal X}_n}

\begin{document}

\title{Two-point coordinate rings for GK-curves}

\author{Iwan M. Duursma} 

\date{April 27 (revised August 16) 2010}

\maketitle

\begin{abstract}
Giulietti and Korchm{\'a}ros presented new curves with the maximal number of points over
a field of size $q^{6}$. Garcia, G{\"u}neri, and Stichtenoth extended the construction to 
curves that are maximal over fields of size $q^{2n}$, for odd $n \geq 3$. The generalized GK-curves
have affine equations $x^q+x = y^{q+1}$ and $y^{q^2}-y = z^r$, for $r=(q^n+1)/(q+1).$
%Giullietti and Fabini describe the one-point coordinate ring for the cases $q=2$ and $q=3.$
%of the GK-curve. 
%In this paper, 
We give a new proof for the maximality of the
generalized GK-curves and we outline methods to efficiently obtain their two-point
coordinate ring.
%The plane curve in the variables $y$ and $z$ was studied by
%Albon, with the first example $y^4-y=x^3$ due to Serre. For all these curves we provide
%properties of the ring of functions that are regular outside two points, including a 
%description of the Weierstrass non-gap sequences. 
%As a consequence, we obtain an effective and efficient construction of two-point codes 
%on the curves. 
\end{abstract} 

\section*{Introduction}

One of the main open problems for curves over finite fields is the classification of maximal curves,
curves that have the maximum number of points in the Hasse-Weil upper bound. Additional motivation for
the problem comes from coding theory since curves with many points can be used to construct long codes 
with good parameters. For many years, all new examples of maximal curves could be derived as
subcovers of the ubiquitous Hermitian curve. Giulietti and Korchm{\'a}ros \cite{GiuKor09} presented an important
new family of maximal curves (GK-curves) that can not be obtained in this way.
The Natural Embedding Theorem says that a curve is maximal over a field 
of size $q^2$ if and only if it is a curve of degree $q+1$ on a Hermitian hypersurface \cite{KorTor01}, \cite{HirKorTor08}. 
GK-curves have a known embedding as a curve on a Hermitian surface \cite{GiuKor09}. \\
Building on an example 
of Serre, Abd{\'o}n, Bezerra and Quoos \cite{AbdBezQuo09} formulated a new family of plane maximal curves. Garcia, 
G{\"u}neri and Stichtenoth \cite{GarGunSti07} construct generalized GK-curves as suitable covers of those plane
curves. For a generalized GK-curve it is not known if it is covered by the Hermitian curve. Nor is it known
how the curve is embedded in a Hermitian hypersurface. \\
%An important tool in the classification of maximal curves is the Natural Embedding Theorem
%\cite{KorTor01}, see also \cite{HirKorTor08}, which says that a curve over a field of size $q^2$ is maximal if and 
%only if it is a curve of degree $q+1$ on a Hermitian hypersurface. For the GK-curve it is known that it is
%not covered by the Hermtian curve and also that it lies on a Hermitian surface \cite{GiuKor09}. For a generalized 
%GK-curve it is not known if it is covered by the Hermitian curve nor how it embeds in a Hermitian hypersruface. \\
In this paper, we provide an elementary proof that generalized GK-curves are maximal. It is well 
known that maximality can be shown by giving, for an arbitrary point over the algebraic closure, 
a carefully chosen hypersurface that intersects the curve only in that point and its conjugates, with 
prescribed multiplicities. For the Hermitian curve, the choice is straightforward. The curve is a
plane curve and for the hypersurface one can choose the tangent at a point. A similar straightforward
choice is available for GK-curves, using their known embedding in a Hermitian surface, but not for
generalized GK-curves. Generalized GK-curves
are defined by two equations in $3-$space and our maximality proof consists of explicitly presenting 
the surface that intersects the curve with the required multiplicities. A different maximality proof 
appeared in \cite{GarGunSti07}. It would be interesting to have yet another proof, along the lines of Weil's classical
paper on curves over finite fields and exponential sums, by expressing the Frobenius eigenvalues as 
exponential sums and then connecting the exponential sums to Gauss sums. \\
The surface that we present in our maximality proof plays a role in the second part of the paper, where
we describe the ring of 
functions on generalized GK-curves that are regular outside two given points. That ring contains a
subring $k[h,h^{-1}]$ of finite index, where $h=0$ is the equation of the surface. We describe a method to obtain 
a basis for the full ring as a free module over the subring $k[h,h^{-1}].$ A possible application, 
not considered in this paper, is the construction of good codes on generalized GK-curves. From the given description
it is straightforward to efficiently construct such codes. \\
To illustrate our methods we present the two-point non-gaps for GK-curves defined over a field of size $q^6$, for
$q=2,3,4.$ This extends earlier results for one-point non-gaps for the cases $q=2,3$ that were obtained in \cite{FanGiu10}. 

\subsection*{Maximal curves}
 
The Hermitian curve $y^q+y = x^{q+1}$ over the field $k$ of $q^2$ elements is special in several
ways. It has $N = q^3+1$ rational points ($q^3$ solutions $(x,y) \in k^2$ and one point at infinity)
and genus $g=q(q-1)/2.$  With these parameters it attains the maximum in the Hasse-Weil bound 
$N \leq q^2+1+2gq$. Moreover, as was shown by Ihara \cite{Iha81}, its genus is maximal among all curves that meet 
the Hasse-Weil upper bound (the so-called maximal curves). Ihara's result has the following
generalization. A curve over the field of $q$ elements with number of points $N > r^m+1$, for $r = \sqrt{q},$ $m \geq 2$, 
has at least one Frobenius eigenvalue $\alpha = r e^{i\theta}$ with $\theta \in (\pi/m,3\pi/m)$ \cite{DuuEnj02}.
A curve is maximal if and only if $\alpha = -r$ and $\theta = \pi$ for all Frobenius eigenvalues.
And in particular, a maximal curve can not have more than $r^3+1$ rational points. The maximum
$N=r^3+1$ is attained only if the curve is Hermitian \cite{RucSti94}. The classification of maximal curves remains a 
major challenge. An important tool in the classification is the Natural Embedding Theorem \cite{KorTor01}, \cite[Theorem 10.22, Remark 10.24]{HirKorTor08}
which says that a curve over $k$ is maximal if and only if it is a curve of degree $q+1$
on a Hermitian hypersurface. Giulietti and Korchm{\'a}ros \cite{GiuKor09} presented a curve that lies on a Hermitian
surface and that can not be obtained as a subcover of the Hermitian curve. It is defined by the
equations (we use the equivalent equations introduced in \cite{GarGunSti07}) 
$x^q+x = y^{q+1}$ and $y^{q^2}-y = z^{q^2-q+1}$ and is maximal over the field of size $q^6$.
Garcia, G{\"u}neri and Stichtenoth \cite{GarGunSti07} extended this to a larger class by proving that more 
generally the curve is maximal over $q^{2n}$ 
for odd $n$ if the exponent for $z$ is replaced with $(q^n+1)/(q+1).$ The proof in \cite{GarGunSti07} uses results
by Abd{\'o}n, Bezerra and Quoos \cite{AbdBezQuo09} for plane curves $y^{q^2}-y=z^{(q^n+1)/(q+1)}$. The genera are
\[
g ~=~ \frac{(q-1)(q^n-q)}{2} \quad \text{and} \quad  g ~=~ \frac{(q-1)(q^{n+1}+q^n-q^2)}{2},
\]
respectively, for the plane curve and for the generalized GK-curve \cite[Proposition 2.2]{GarGunSti07}. To prove maximality it is then verified 
that the number of points $N$ meets the Hasse-Weil bound. In this paper we give a
different proof that shows that for a generalized GK-curve all Frobenius eigenvalues are equal to
$-q^n.$ Whereas the maximality proof in \cite{GarGunSti07} uses the results in \cite{AbdBezQuo09}, our proof
implies those results. Our proof is elementary and consists of two steps. We look
for a special function $h$ that intersects the curve with high multiplicity in a given point. 
The two steps are (1) find suitable equations for $h$ and (2) find the function $h$ as a solution to the equations. 
The main difficulty is to stay away from computing large resultants and to introduce suitable short cuts
in the computations. The structure of the curves is nice enough that this is indeed possible
and the computations remain perfectly manageable. 

\subsection*{Rings of regular functions}

%The different proof for the maximality of the generalized GK-curve appears as a side result of the paper. 
%Our main goal 
A second goal of the paper is to describe the functions on the generalized GK-curve in a way that makes it feasible to 
use the curve for the construction of good linear codes. The first results in this direction were obtained
by Fanali and Giulietti \cite{FanGiu10}, with a description of the functions regular outside a rational point
on the GK-curve, for the cases $q=2$ and $q=3$. The functions that we are looking for generalize
similar functions for the Hermitian curve. For the Hermitian curve with equation $x^q+x=y^{q+1}$
(note that we follow the notation for the GK-curve and not the common notation $y^q+y=x^{q+1}$ for the
Hermitian curve), one-point Hermitian codes use as ring of functions
%with choice of 
%point the point at infinity,
\[
k[x,y] ~=~ \text{ the free $k[x]$ module with basis $\{ 1, y, \ldots, y^q  \}$. }
\]
The functions in the ring have no poles outside the point of infinity. 
Two-point codes use a larger ring of functions with the possibility of poles at a second point. For the choice
of $(0,0)$ as the second point, the ring extends to 
%\begin{equation} \label{eq:origin}
\[
k[x,x^{-1},y] =  \text{ the free $k[x,x^{-1}]$ module with basis $\{ 1, y, \ldots, y^q  \}$. }
\]
%\end{equation}
More generally, for a choice $(\alpha,\beta)$ as second point, and for $(a,b) = (\alpha^q,\beta^q)$
we obtain the ring
\begin{multline*} %\label{eq:ab}
k[x+a-by,(x+a-by)^{-1},y-b^q] \\ = \text{ the free $k[(x+a-by),(x+a-by)^{-1}]$ module } \\
\text{ with basis $\{ 1, y-b^q, \ldots, (y-b^q)^q = y^q-b \}$.}
\end{multline*}
%Motivation for the choices is given in Section \ref{}.
The surface $h=0$ that we will determine in the next two sections generalizes the line $x+a=by$. 
As outlined above, the function $h$ can be used to prove the maximality of the generalized GK-curves.
For the Hermitian curve, the automorphism group acts transitively on ordered pairs of rational points and
we can always assume that the first point is the point at infinity and the second point is the origin 
$(0,0)$. For the generalized GK-curves, this is not the case. For the GK-curve itself, the rational points 
divide into two orbits \cite{GiuKor09}. We claim that the ring of functions regular outside the point at infinity and the
origin $(0,0,0)$ is the ring 
\[
k[x,x^{-1},y,z] ~=~ \text{ the free $k[x,x^{-1}]$ module with basis $\{ y^i z^j : 0 \leq i < q+1, 0 \leq j < r  \}$. }
\]
It is known (\cite[Proposition 1]{GiuKor09}, see also \cite[Section IV]{FanGiu10}) that the ring of functions regular 
outside the point at infinity is generated by the three functions $x, y$ and $z$. To find the full ring it suffices
to invert one function that has poles only at infinity and that vanishes only in $(0,0,0)$. The function $x$ qualifies.  
In Section \ref{S:cr}, we give a description when the second point is not in the orbit of the origin $(0,0,0)$.
%as in Equation (\ref{eq:ab}).
%In a separate section we describe the construction of a suitable basis that generalizes the role
%of the basis $\{ 1, y-b^q, \ldots, y^q-b \}$. With the given description of the coordinate rings we
%believe it is feasible to apply efficient algorithms developed for the Hermitian curve to
%generalized GK-curves. General methods to construct the free basis were outlined in our 
%Oberwolfach presentation at the 2007 Coding Theory meeting.       

\section{The surface $x+a=by+Z$ in implicit form} \label{S:implicit}

The generalized GK-curve $\Cn$ is defined, for an odd integer $n$, and for $r = (q^n+1)/(q+1)$, 
by the pair of equations
\[
\begin{cases} x^q+x = y^{q+1} \\ y^{q^2}-y = z^r \end{cases}
\]
The curves were formulated in \cite{GarGunSti07} and it was shown there
that the curve $\Cn$ has genus 
\[
g(\Cn) =  \frac{(q-1)(q^{n+1}+q^n-q^2)}{2} \\
%g(\Xn) ~&=~ \frac{(q-1)(q^n-q)}{2} 
\]
and that it has the maximum number $N = q^{2n} + 1 + 2g q^n$ of ${\mathbb F}_{q^{2n}}$-rational points. 
A curve is maximal over ${\mathbb F}_{q^{2n}}$ if and only if, for an arbitrary point $P$ on the curve and for a rational point $P_0$,
the following {\emph {fundamental linear equivalence}} holds,
\[
q^n P + \Phi (P) \sim (q^n+1) P_0.
\]
Here $\Phi$ denotes the $q^{2n}$-Frobenius morphism that raises the coordinates of a point to the power $q^{2n}.$
The equivalence implies that the $q^{2n}$-Frobenius action on the Jacobian of $\Cn$ (more precisely on a $\ell$-adic Tate module
for the Jacobian) has unique eigenvalue $-q^n$, and therefore that the curve has the maximum number of $q^2+1+2gq^n$ 
rational points over ${\mathbb F}_{q^{2n}}$. We will present, for any point $P=(\alpha,\beta,\gamma)$ on $\Cn$, possibly with coordinates 
in an extension field, a polynomial $h_P \in k[x,y,z]$ with divisor 
\[
(h_P) = q^n(\alpha,\beta,\gamma)+(\alpha^{q^{2n}},\beta^{q^{2n}},\gamma^{q^{2n}})
      -(q^n+1) \infty,
\]
where $\infty$ denotes the common pole of $x, y$ and $z$. This proves the fundamental linear equivalence and hence the
maximality of $\Cn$. The function $h$ uses coefficients
$a = \alpha^{q^n},$ $b = \beta^{q^n},$ and $c = \gamma^{q^n}.$ Clearly,
$(a,b,c)$ is a point on $\Cn$.
\[
\begin{cases} a^q+a = b^{q+1} \\ b^{q^2}-b = c^r \end{cases}
\]
Ignoring for the moment the second equation and the variable $z$, we are
left with the Hermitian curve $x^q+x=y^{q+1}$ and a point $(a,b)$ on the Hermitian curve, 
i.e. $a^q+a = b^{q+1}.$ \\

For $a = \alpha^q$ and $b = \beta^q$, the
line $x + a = by$ intersects the curve in $(\alpha, \beta)$ ($q$ times) and
$(\alpha^{q^2}, \beta^{q^2})$ (multiplicity one). This is the classical 
proof that the Hermitian curve of degree $q+1$ is maximal over the field
of $q^2$ elements. For the computation of the intersection divisor of
the line $x+a = by$, we start with the three equations
\[
\begin{cases} x^q+x = y^{q+1} \\ a^q+a = b^{q+1} \\ x+a = by  \end{cases}
\]
After elimination of $x$ and $a$ we find
\[
y^{q+1}+b^{q+1}-(by)^q-by = (y^q-b)(y-b^q) = 0.
\]
For $b = \beta^q$, the solutions are $y=\beta$ ($q$ times) and
$y=\beta^{q^2}$ (multiplicity one). \\

We want to extend the argument, which
is well known for the Hermitian curve, to the curve $\Cn$. For the curve 
$\Cn$, we bring in the second equation and consider the system of equations
\[
\begin{cases} x^q+x = y^{q+1} \\ a^q+a = b^{q+1} \\
              y^{q^2}-y = z^r \\ b^{q^2}-b = c^r \\
              x+a = by+Z \end{cases}
\]
The purpose is to replace $Z$ with a suitable polynomial such that the 
surface $x+a = by+Z$ vanishes in $(\alpha, \beta, \gamma)$ ($q^n$ times) 
and $(\alpha^{q^{2n}}, \beta^{q^{2n}},\gamma^{q^{2n}})$ (multiplicity one),
for $a = \alpha^{q^n}, b = \beta^{q^n},$ and $c = \gamma^{q^n}.$ To this end
we eliminate $x$ and $a$ as well as $y$ and $b$ from the equations, which
will lead to an expression
\[
F(Z,z,c) = 0,
\]
for a polynomial $F$ that is symmetric in $z$ and $c$. 
%As an example, for $q=2$,
%\[
%{Z}^{8}+{Z}^{4})+{Z}^{5}+{Z}^{4}+{Z}^{3}+Z+ \left( {Z}^{2}+Z \right) 
% \left( zc \right) ^{r}+ \left( {z}^{2\,r}+{c}^{r} \right)  \left( {z}
%^{r}+{c}^{2\,r} \right) = 0
%\]
The next step will
be to choose $Z=Z(z,c)$ such that 
\[
F(Z(z,c),z,c) = (z^{q^n}-c)(z-c^{q^n}).
\]
For such a choice of $Z$, the hypersurface $x+a = by+Z$ intersects the
curve $\Cn$ in points $(x,y,z)$ with $z$ a root of 
\[
(z^{q^n}-c)(z-c^{q^n}) = (z^{q^n}-\gamma^{q^n})(z-\gamma^{q^{2n}}) =
(z-\gamma)^{q^n}(z-\gamma^{q^{2n}})=0.
\]
%We will see that Equation () takes the form $G(Z(cz),cz) = 0$. For $q=2$,
%\[
%{Z}^{8}+{Z}^{6}+{Z}^{5}+{Z}^{4}+{Z}^{3}+Z+ \left( {Z}^{2}+Z \right) 
% \left( zc \right) ^{r}+ \left( {cz}^{r}+{cz} \right)  \left( {cz}
%^{qr}+{cz} \right) = 0.
%\]
%For general $q$. 
After elimination of $x$ and $a$ we find
\[
\begin{cases} (y^q-b)(y-b^q) = Z^q+Z \\
              y^{q^2}-y = z^r \\ b^{q^2}-b = c^r 
              \end{cases}
\]
We express the left hand sides in terms of the new
variables $u = y^q-b$ and $v = y-b^q$. And we replace the last 
two equations with two new equations.
\[
\begin{cases} uv = Z^q+Z \\
              (u^q-v)(u-v^q) = (cz)^r \\ 
              (u^{q^2}-u)(v^{q^2}-v) = (z^{qr}-c^r)(z^r-c^{qr})
              \end{cases}
\]
%With $(q+1)r = q^l+1$, and for $t=cz$, we have
%\begin{align*}
%(z^{qr}-c^r)(z^r-c^{qr}) &= (z^{q^l}-c)(z-c^{q^l}) + cz + (cz)^{q^l} - (cz)^r - (cz)^{rq} \\
% &= (z^{q^l}-c)(z-c^{q^l}) + (t^r-t)(t^{qr}-t)/t.
%\end{align*}
Elimination of $u$ and $v$ is straightforward. Note that all left sides are 
symmetric polynomials in $u$ and $v$. The symmetric polynomial $(X^{q-1}-1)(Y^{q-1}-1)$ 
can be written as a polynomial $F(X+Y,XY)$ in $X+Y$ and $XY$.
\begin{align*}
(X^{q-1}-1)(Y^{q-1}-1) &= \prod_{\zeta^{q-1}=1} (X-\zeta)(Y-\zeta) \\                         
                     &= \prod_{\zeta^{q-1}=1} (XY-\zeta(X+Y)+\zeta^2) =: F(X+Y,XY). 
\end{align*}
%\begin{align*}
%X^{q-1}+Y^{q-1} &= (XY)^{q-1}+1-(X^{q-1}-1)(Y^{q-1}-1) \\
%              &= (XY)^{q-1}+1-\prod_{\zeta^{q-1}=1} (X-\zeta)(Y-\zeta) \\                         
%              &= (XY)^{q-1}+1-\prod_{\zeta^{q-1}=1} (XY-\zeta(X+Y)+\zeta^2). 
%\end{align*}
Thus, for $X=u^{q+1}, Y = v^{q+1},$ %and for $w = uv$,
\begin{align*}
(z^{qr}-c^r)(z^r-c^{qr}) &= (u^{q^2}-u)(v^{q^2}-v) \\
                       &= (uv)(X^{q-1}-1)(Y^{q-1}-1) \\
                       &= (uv)F(X+Y,XY) \\
                       &= (uv) F((cz)^r+uv+(uv)^q,(uv)^{q+1}).
\end{align*}
%For $q=2$, $(X-1)(Y-1)=1-(X+Y)+XY$. For this choice of $q$, and for any exponent $r$, 
%\[
%(z^{2r}-c^r)(z^r-c^{2r}) = w (t^r+1+w+w^2+w^3).
%\]
%For $w=Z^q+Z$ this returns the previous equation. For larger $q$, the polynomials $F$ grow
%quickly in size and we introduce some shortcuts to describe both the equation for $w$ and its solution. \\
  
With $(q+1)r = q^n+1$, we have
\begin{align*}
(z^{qr}-c^r)(z^r-c^{qr}) &= (z^{q^n}-c)(z-c^{q^n}) + cz + (cz)^{q^n} - (cz)^r - (cz)^{rq} \\
 &= (z^{q^n}-c)(z-c^{q^n}) + ((cz)^r-cz)((cz)^{qr-1}-1).
\end{align*}
Thus we are looking for $Z$ such that, for $uv=Z^q+Z$,
\[
(cz)(uv) F((cz)^r+uv+(uv)^q,(uv)^{q+1}) = ((cz)^r-cz)((cz)^{qr}-cz).
\]
%Note that $F(uv+(uv)^q,(uv)^{q+1})=((uv)^{q-1}-1)((uv)^{q^2-q}-1).$ For $t=uv=Z^q+Z \pmod{cz}^r$,
%\[
%t(t^{q-1}-1)(t^{q^2-q}-1) = cz
%\] 
%with solution $t = \sum_{l \geq 3, l \text{odd}} \sum_{i \geq 0} (cz)^{q^ir}.$
%\[
For $t = cz,$ and for $w=uv=Z^q+Z$,
\[ 
tw F(t^r+w+w^q,w^{q+1}) = (t^r-t)(t^{qr}-t),
\]
where $F(X+Y,XY) = (X^{q-1}-1)(Y^{q-1}-1).$
  
\begin{theorem} \label{T:eqnsurface}
For odd $n \geq 1$, let $\Cn$ be the generalized GK-curve over ${\mathbb F}_{q^{2n}}$, defined by the equations
\[
\begin{cases} x^q+x = y^{q+1} \\
              y^{q^2}-y = z^{r} 
\end{cases}
\]
where $r=(q^n+1)/(q+1)$. Let
\[
tw F(t^r+w+w^q,w^{q+1}) = (t^r-t)(t^{qr}-t)
\]
for $w = Z^q+Z, t = cz$ and for $F(X+Y,XY) = (X^{q-1}-1)(Y^{q-1}-1).$ 
Then, for any point $(\alpha,\beta,\gamma) \in C_n$, and for $(a,b,c)=( \alpha^{q^n}, \beta^{q^n}, \gamma^{q^n})$, the surface $x+a = by + Z$ intersects the curve $C_n$ in
\[
(\alpha, \beta, \gamma) ~(\text{$q^n$ times}) 
\quad \text{ and } \quad (\alpha^{q^{2n}}, \beta^{q^{2n}},\gamma^{q^{2n}}) ~(\text{multiplicity one}).
\]
\end{theorem}

\begin{corollary} Let $\Xn$ be the plane curve over ${\mathbb F}_{q^{2n}}$, defined by the equation $y^{q^2}-y = z^{r}.$ 
Then, for any point $(\beta,\gamma) \in \Xn$, and for $(b,c)=( \beta^{q^n}, \gamma^{q^n})$, the curve $(y^q-b)(y-b^q) - w$ intersects the curve $\Xn$ in
\[
(\beta, \gamma) ~(\text{$q^n$ times}) 
\quad \text{ and } \quad (\beta^{q^{2n}},\gamma^{q^{2n}}) ~(\text{multiplicity one}).
\]
\end{corollary}

\begin{proof}
The equation of the curve follows by taking the trace of $x+a - (by + Z)$,
\begin{multline*}
%(x+a-yb-Z)^q+(x+a-yb-Z) = 
(x^q+x)+(a^q+a)-(yb)^q-(yb)-(Z^q+Z) = \\
y^{q+1}+b^{q+1}-(by)^q-by-(Z^q+Z) = (y^q-b)(y-b^q)-w.
\end{multline*}
\end{proof}

\section{The surface $x+a = by + Z$ in explicit form} \label{S:explicit}

For a prime power $q=p^m$, we define two formal sums $\hat f$ and $\hat g$ in the variable $t$ with coefficients in a field
of characteristic $p$.
\begin{align*}
&\hat f = \sum_{n \geq 1} t^{r_n}, \qquad r_n=(q^{2n-1}+1)/(q+1), \\
&\hat g = \sum_{n \geq 0} t^{s_n}, \qquad s_n=(q^{2n}-1)/(q+1).
\end{align*}
%$\hat f$ denote the formal sum of monomials $t^r$ with $r=(q^{2n-1}+1)/(q+1)$, for $n \geq 1$, and let $\hat g$ denote the formal sum of monomials $t^s$ %with $s=(q^{2n}-1)/(q+1)$, for $n \geq 0.$ So that 
So that
\begin{align*}
{\hat f} &= t+ t^{q^2-q+1} + t^{q^4-q^3+q^2-q+1} + \cdots, \\
{\hat g} &= 1 + t^{q-1} + t^{q^3-q^2+q-1} + \cdots. 
\end{align*}
As power series in the variable $t$, $\hat f$ and $\hat g$ are uniquely determined
by their initial terms, $\hat f = t + \text{(higher order terms)}$ and $\hat g = 1 + \text{(higher order terms)}$,
and by the relations 
\[
\hat f = t {\hat g}^q,  \quad t {\hat g} = t+ {\hat f}^q.
\]
It follows that
%\begin{align*}
\[
t^{q-1} ({\hat f}-t) = {\hat f}^{q^2}, \quad ({\hat g}-1) = t^{q-1} {\hat g}^{q^2}. 
%\end{align*}
\]
We will also use 
\[
{\hat f} {\hat g} = {\hat f} + ({\hat f} {\hat g})^q = {\hat f} + {\hat f}^q + {\hat f}^{q^2} + \cdots.
\]
For a given $n$, let $r=r_n$ and $s=s_n$, so that $s=qr-1$ and $r+s=q^{2n-1}.$  Let $f$ be the sum of monomials in $\hat f$ of degree less than $r$ and
let $g$ be the sum of monomials in $\hat g$ of degree less than $s$. 

\begin{lemma} The equation 
\[
tw F(t^r+w+w^q,w^{q+1}) = (t^r-t)(t^{qr}-t),
\]
where $F(X+Y,XY) = (X^{q-1}-1)(Y^{q-1}-1),$ has a polynomial solution $w=fg$.
\end{lemma}
 
\begin{proof}
The polynomials $f$ and $g$ satisfy
\begin{align*}
t^{q-1} (f + t^r - t) = f^{q^2}, \\
(g + t^s - 1) = t^{q-1} g^{q^2}. 
\end{align*}
And therefore,
\[
\begin{array}{llllllll}
t^{r}-t    &= (X^{q-1}-1) f     &\text{ for } ~t X = {f}^{q+1}, \\
t^s-1      &= (Y^{q-1}-1) g     &\text{ for } ~Y = t {g}^{q+1}.
\end{array}
\]
Moreover, 
\[
f + t^r = t g^q, \qquad t g = t + f^q.
\]
Let $w = fg,$ so that $XY = w^{q+1}.$ Then
\begin{align*}
&X = f(g-1), \quad Y = (f+t^r)g, \\
&w = fg,  \quad w^q = XY/w = (g-1)(f+t^r), 
\end{align*}
and $X+Y = w +w^q + t^r.$ Thus $w=fg$ is a solution.
\end{proof} 

For the surface $x+a = by + Z$ in Theorem \ref{T:eqnsurface} we need $Z$ such that $Z^q+Z = w$. With the lemma
we may obtain $Z$ as a solution to $Z^q+Z = fg$. For a given $n$, $f = t^{r_1} + \cdots + t^{r_{n-1}}$ and
$g = t^{s_0} + \cdots + t^{s_{n-1}}.$ The product $fg$ is the sum of $(n-1)n$ terms of the form $t^{r_i+s_j}$.
It is helpful to picture the sums $r_i+s_j$ in an addition table. For $n=4$, 
\[
\begin{array}{l|llll} &s_0 &s_1 &s_2 &s_3  \\ \hline
r_1 &1 &q \\
r_2 &  &q^2 &q^3 \\
r_3 &  &    &q^4 &q^5 \\
\end{array}
\]
Only the entries near the diagonal have been filled in. They follow from the relations 
$r_i + s_{i-1} = q^{2i-2}$ and $r_i + s_i = q^{2i-1},$ for $i \geq 1.$ In general, 
for $0 \leq j < i < n$, 
\[
r_{j+1}+s_{i} = q (r_{i}+s_{j}).
\]
Thus the entries above the diagonal are $q$-multiples of the entries below the diagonal.
And we can choose $Z = \sum_{0 \leq j < i < n} t^{r_i + s_j}.$

%\begin{align*}
%&\hat f = \sum_{n \geq 1} t^{r_n}, \qquad r_n=(q^{2n-1}+1)/(q+1), \\
%&\hat g = \sum_{n \geq 0} t^{s_n}, \qquad s_n=(q^{2n}-1)/(q+1).
%\end{align*}

\begin{example} Let ${\cal C}$ be the GK-curve defined by $x^q+x=y^{q+1}, y^{q^2}-y=z^{r}$, $r=q^2-q+1$, over ${\mathbb F}_{q^{6}}.$
Then $w=t^q+t$ and $Z=t$. The equation of the hypersurface is $x+a-by-cz = 0$. The equation appears in \cite{FanGiu10} and in this
special case can be obtained directly as the tangent plane to the Hermitian surface $x^{q^3}+x = y^{q^3+1} + z^{q^3+1}.$ It follows
from the proof of the Natural Embeding Theorem \cite{KorTor01}, see also \cite[Theorem 10.22, Remark 10.24]{HirKorTor08} that the 
latter approach to find the equation of the hypersurface applies more generally to all maximal curves with a known embedding in a 
Hermitian hypersurface.
\end{example}
 
\begin{example} Let ${\cal C}$ be the maximal curve defined by $x^2+x=y^3, y^4-y=z^{11}$ over ${\mathbb F}_{2^{10}}.$
Let $(\alpha,\beta,\gamma)$ be a rational point with $\gamma \neq 0$ and let $(a,b,c)=(\alpha^{32},\beta^{32},\gamma^{32}).$
For $q=2$, $\hat f = t+t^3+t^{11}+t^{43}+\cdots$ and $\hat g = 1+t+t^5+t^{21}+t^{85}+\cdots$. For the given curve we use 
$r=11$, $s=21$, $r+s=32$, $f= t+t^3$ and $g=1+t+t^5$. The rational function
\[
(y^2-b)(y-b^2)-f(cz)g(cz)
\]
for the plane curve ${\cal X}$ with $y^4-y=z^{11}$ has a pole of order $2^5+1$ at infinity and a zero of order $2^5+1$ at $(\beta,\gamma)$. 
With $fg = (t+t^3+t^4)+(t+t^3+t^4)^2$ and $Z=t+t^3+t^4$, the rational function
\[
x+a - by - Z(cz)
\]
for the curve ${\cal C}$ has a pole of order $2^5+1$ at infinity and a zero of order $2^5+1$ at $(\alpha,\beta,\gamma)$.
\end{example}
 
%\]
%Also, $X = (g^q-t^{r-1})f^q = (fg)^q  - t^r (g-1).$ $Y = (t+f^q)g^q = (fg)^q + (f + t^r).$  
%\[
%X+Y = f(g-1) + t g^{q+1} = fg - f + (t + f^q) g^q =  fg +(fg)^q + t^r = t^r + w + w^q.
%\]
%\[
%q=4, l=5: cz+{c}^{1024}{z}^{1024}+{c}^{1025}+{z}^{1025} = F(cz+{c}^{4}{z}^{4}+{c}^{13}{z}^{13}+{c}^{16}{z}^{16}+{c}^{52}{z}^{52}+{
%c}^{64}{z}^{64},c^{205} z^{205}).
%\]
%\[
%(z^{q^l}-c)(z-c^{q^l}) + cz + (cz)^{q^l} = F(w,uv) + uv + (uv)^q.
%\]
\begin{remark} The equations that we used in the previous section are formulated for pairs of points $(x,y,z)$ and $(a,b,c)$ on the same
curve, and the equations express a correspondence between two copies of the same curve. For a curve with function field    
$K/k$ defined over a finite field of size $q$, let $F : K \longrightarrow K$ be the purely inseparable map $F(x) = x^q$
and let $\psi : K \longrightarrow K$ be the map $\psi(x) = x^{q^2}-x$. Using two copies of this map, we find a map
$\psi \times \psi : K \times K \longrightarrow K \times K$ with a factorization $\psi \times \psi = \phi \circ \phi,$
for $\phi : K \times K \longrightarrow K \times K$ such that $\phi(x,y) = (x^q-y,x-y^q).$ In matrix form, 
\[
\left( \begin{array}{cc} F^2-1 &0 \\ 0 &F^2-1 \end{array} \right) ~=~
\left( \begin{array}{cc} F &-1 \\ 1 &-F \end{array} \right)
\left( \begin{array}{cc} F &-1 \\ 1 &-F \end{array} \right).
\]
If we denote by $p : K \times K \longrightarrow K$ the product $p(x,y)=xy$ then we can write the left sides of our equations
as
\begin{align*}
p \circ \phi(y,b) &~=~ uv \\
p \circ \phi^2(y,b) &~=~ (u^q-v)(u-v^q) \\
p \circ \phi^3(y,b) &~=~ (u^{q^2}-u)(v^{q^2}-v). 
\end{align*}
\end{remark}

\section{Two-point coordinate rings} \label{S:cr}

For the Hermitian curve, with equation $x^q+x=y^{q+1}$, the functions that are regular except possibly at infinity
are the polynomials in $x$ and $y$. Using the equation of the curve, each such polynomial can be represented uniquely
as a polynomial of degree at most $q$ in $y$ with coefficients in $k[x],$  
\[
k[x,y] ~=~ \text{ the free $k[x]$ module with basis $\{ 1, y, \ldots, y^q  \}$. }
\]
The larger ring of all functions that are regular except possibly at infinity or at the origin $(0,0)$ is 
\[
k[x,x^{-1},y] ~=~  \text{ the free $k[x,x^{-1}]$ module with basis $\{ 1, y, \ldots, y^q  \}$. }
\]
The function $x$ has divisor $(q+1)((0,0)-\infty)$. It has a zero of order $q+1$ at the origin,
a pole of order $q+1$ at infinity, and no other zeros or poles. The function $y$ has poles only at
infinity, of order $q$, and it has $q$ zeros of order one including a zero at the origin. Note that
as $f$ runs through $\{ 1, y, \ldots, y^q \}$, the pole order of $f$ at infinity and the order of 
vanishing of $f$ at the origin both run through all the residue classes modulo $q+1$. This assures
that the functions $1, y, \ldots, y^q$ are independent over $k[x].$ \\

For the curve with equation $y^{q^2}-y = z^r$, for $r = (q^n+1)/(q+1)$, $n$ odd, the functions that
are regular outside infinity are again the polynomials in $y$ and $z$. The ring of functions that are 
regular except possibly at infinity or at the origin $(0,0)$ is 
\[
k[y,y^{-1},z] =  \text{ the free $k[y,y^{-1}]$ module with basis $\{ 1, z, \ldots, z^{r-1}  \}$. }
\]
The main properties are the same as before. For a choice of $(\beta,\gamma)$ as second point such that 
$\gamma=0$, we replace $y$ with $y-\beta$. The function $y - \beta$ has a pole of order $r$ at infinity
and a unique zero at $(\beta,0)$ of order $r$. However when we choose as second point on the curve a
rational point $(\beta,\gamma)$ with $\gamma \neq 0$, then we need the function $h$ of the previous
section. For $b = \beta^{q^n}$, $c = \gamma^{q^n}$,
\[
(h) = (q^n+1)((\beta,\gamma)-\infty), \qquad  \text{ for }~h = (y^q-b)(y-b^q) - f(cz)g(cz).
\]
The ring of functions regular outside infinity and $(\beta,\gamma)$ becomes
\[
k[h, h^{-1}, y, z] = \text{ the free $k[h,h^{-1}]$ module with basis $\{ y^i z^j : 0 \leq i < q+1, 0 \leq j < r \}$. }
\]
While this gives a correct description of the ring it is not quite in the right form to recognize the 
$k$-subspaces that are needed for the construction of two-point codes. For that purpose we order the
monomials $y^i z^j$ by increasing pole order at infinity. And we replace each monomial $y^i z^j$ with a
polynomial $f_{i,j}$ that has $y^i z^j$ as leading monomial and whose vanishing order at the second
point $(\beta,\gamma)$ is maximal.
\begin{align*}
k[h, h^{-1}, y, z] &= \text{ the free $k[h,h^{-1}]$ module with} \\
                   &\qquad \text{basis $\{ f_{i,j} = y^i z^j + \cdots  : 0 \leq i < q+1, 0 \leq j < r \}$. }
\end{align*}

The function $h$ with $(h) = (q^n+1)((\beta,\gamma)-\infty)$ provides useful relations between the Weierstrass
semigroups at the points $(\beta,\gamma)$ and $\infty$. We make a small digression to describe these relations
for general points $P$ and $Q$ and for a function $h$ with divisor $m(P-Q)$. Applying the relations with 
$P=\infty$ and $Q=(\beta,\gamma)$ we will then be able to determine properties of the Weierstrass semigroup at $Q$
using known properties of the Weierstrass semigroup at $P$. We will also be able to motivate the choice of the
modified free basis $\{ f_{i,j} \}$ for the two-point coordinate ring.
Properties of pairs of Weierstrass nongaps 
are considered in \cite{Kim94}, \cite{Mat01}. In those papers a pair of nonnegative integers
$(a,b)$ is called a nongap for $(P,Q)$ if there exists a rational function $f$ with pole divisor $aP+bQ$.
In \cite[Definition 1]{BeeTut06}, this definition is relaxed to arbitrary pairs of integers, and a pair of integers 
$(a,b)$ is a nongap if there exists a function $f$ with no poles outside $P$ and $Q$ such that
$-\ord_P(f) = a, -\ord_Q(f) = b$. At least one of $a, b$ is nonnegative. If $a, b \geq 0$ the nongap is a classical
pole divisor. If $a \geq 0$ and $b < 0$ then $f$ has a pole of order $a$ at $P$ and vanishes to the order $-b$ at $Q$,
with a similar interpretation when $b \geq 0$ and $a < 0$. \\

For two points $P$ and $Q$, fix an integer $m$ such that $mP \sim mQ$, and let $h$ be a function with divisor $(h)=m(P-Q)$.  
Clearly $m$ is a common Weierstrass nongap for $P$ and $Q$. We call a Weierstrass $P$-nongap $a$ minimal  
if it is the smallest nongap in the residue class $a + m {\mathbb Z}$. Similarly for a Weierstrass $Q$-nongap $b$.
The Weierstrass semigroups for $P$ and $Q$ are determined by their $m$ minimal nongaps. The number of minimal Weierstrass 
nongaps in the interval $(im-m,im]$ is the same for $P$ and for $Q$, for any integer $i$. This is easy to see using
that $\dim L (imP) = \dim L(imQ)$, for every integer $i$. We show that moreover there exists a natural bijection
between the minimal nongaps in $(im-m,im]$ for $P$ and those for $Q$. The bijection is the restriction of a
bijection defined on all integers in \cite[Definition 13, Proposition 14(i)]{BeeTut06}, \cite[Theorem 8.5]{DuuPar08}. 

\begin{proposition} \label{P:bijection}
For a Weierstrass $P$-nongap $a \in (im-m,im]$, let $b$ be maximal such that there exists a function $f \in L(aP)$ 
with precise pole order $a$ at $P$ and precise vanishing order $b$ at $Q$. Then $a$ is minimal if and only if $b \in [0,m).$ 
In that case, $b' = im - b \in (im-m,im]$ is a minimal Weierstrass $Q$-nongap with maximal vanishing order $a' = im - a \in [0,m)$
at $P$.
\end{proposition}

\begin{proof}
Assume that $a$ is minimal. Clearly, $b \geq 0$. On the other hand $b < m$. For otherwise $f/h$ would be a function with poles 
only at $P$ of order $a-m$, contradicting minimality of $a$. 
The function $f /h^i$ has a pole only at $Q$ of order $b' = im-b \in (im-m,im].$ The relation between $a$ and $b$ is 
characterized by 
\[
L(aP-bQ) \neq L((a-1)P-bQ) \quad \text{and} \quad  L(aP-(b+1)Q) = L((a-1)P-(b+1)Q,
\]
which is equivalent to the combination 
\[
L(aP-bQ) \neq L(aP-(b+1)Q) \quad \text{and} \quad  L((a-1)P-bQ) = L((a-1)P-(b+1)Q.
\]
For $a' = im - a$ and $b' = im -b$, the latter becomes
\[
L(b'Q-a'P) \neq L((b'-1)Q-a'P) \quad \text{and} \quad L(b'Q-(a'+1)P) = L((b'-1)Q-(a'+1)P).
\]
This shows that $a'$ is the maximum vanishing order at $P$ for a function with precise pole order $b'$ at $Q$. 
Since $a \in (im-m,im]$, $a' \in [0,m)$. But then $b' \in (im-m,im]$ is minimal. 
\end{proof}

We return to the special case $P=\infty$ and $Q=(\beta,\gamma)$, with $m=q^n+1$. The $q^n+1$ minimal $P$-nongaps are 
$\{ ir+jq^2 :  0 \leq i < q+1, 0 \leq j < r \}.$ With the proposition the minimal $Q$-nongaps can be determined from the
maximal vanishing orders at $Q$ of functions with poles only at $P$. In the process of finding the maximal vanishing orders
we update the free basis of monomials $\{ y^i z^j \}$ to the free basis $\{ f_{i,j} \}$. The new basis will be useful when 
we need to find a basis for a given vector space $L(aP+bQ)$. If $L((a-1)P+bQ)$ is properly contained in $L(aP+bQ)$ then
we choose $f \in L(aP+bQ) \backslash L((a-1)P+bQ)$ as follows. Let $a - km$ be a minimal nongap and let $g$ be the function
in the free basis $\{ f_{i,j} \}$ with precise pole order $a-km$ at $P$. Then $g$ vanishes with maximal order at $Q$ and we can 
choose $f = g h^k.$ \\

Replacing the monomial $y^i z^j$ with a polynomial $f_{i,j} = y^i z^j + \cdots$ is essentially a process of Gaussian elimination on a square
matrix of size $q^n+1$ over $k$. Namely for each monomial $y^i z^j$ we consider its development as a power
series in a local parameter $t$ at the point $(\beta,\gamma)$ as follows. The functions $y-\beta$ and $z-\gamma$ each vanish to the order one in 
$(\beta,\gamma)$. We set $z = \gamma (t + 1)$, so that the new variable $t$ vanishes to the order one in $(\beta,\gamma)$. After fixing $t$
as a local parameter we express $y$ as a power series in $t$. Note that $(y-\beta)^{q^2} - (y-\beta) = z^r - \gamma^r = \gamma^r ((1+t)^r - 1).$
If we let $T = (1+t)^r-1$ and $c = \gamma^r$ then
\[
y  = \beta - cT - (cT)^{q^2} - (cT)^{q^4} - \cdots. 
\]
For an arbitrary monomial $y^i z^j$ we find its power series in $t$ by substituting the series for
$y$ and for $z$. We associate to each monomial a
vector of length $q^n+1$ whose coordinates are the coefficients of its power series modulo $t^{q^n+1}$. The Gaussian elimination comes with the
restriction that only previous rows can be used to clear entries in the current row. The computations reduce
significantly if we fill in the rows one at a time, and each time a row is needed we fill it not with the
development of $y^i z^j$ but with the development of either $f_{i-1,j} y$ or $f_{i,j-1} z.$ In that case
each new row requires at most $q+1$ operations to be updated to a polynomial $f_{i,j}$. This is very
similar to the obtained improvements in the Berlekamp-Massey algorithm. \\

For the curves $y^4 - y = z^3$ over ${\mathbb F}_{2^6}$, 
$y^9 - y = z^7$ over ${\mathbb F}_{3^6}$, and
$y^{16}-y = z^{13}$ over ${\mathbb F}_{4^6}$, 
%\[
%\begin{array}{lll}
%y^4 - y = z^3 &~~ &\text{over ${\mathbb F}_{2^6}$}, \\
%q^9 - y = z^7 &~~ &\text{over ${\mathbb F}_{3^6}$}, \\
%y^16-y = z^{13} &~~ &\text{over ${\mathbb F}_{4^6}$}, \\
%\end{array}
%\]
the results are summarized in tables that give the positions of the pivots after
the Gaussian elimination is completed. The matrices are of size $9 \times 9$, $28 \times 28$ and $65 \times 65$
respectively. Rows in the matrices correspond to monomials $y^i z^j$ and are ordered by increasing pole order 
at infinity (i.e. $1 < y < z < y^2 < yz < z^2 < y^3 < \cdots$). For each monomial $y^i z^j$, the table lists
the maximal vanishing order at a point $(\beta,\gamma)$ with $\gamma \neq 0$, for a polynomial $f_{i,j}$ with 
leading monomial $y^i z^j$. A vanishing order of $m$ corresponds to a power series with leading term $t^m$
and to a row with a pivot in the $m+1$-st column. The polynomials that result after the Gaussian elimination
is completed are independent, with distinct vanishing orders in the range $0$ to $q^n$. \\
%It follows from
%properties of pairs of Weierstrass nongaps \cite{Kim94}, \cite{Mat01}, and in particular from results in 
%\cite{BeeTut06}, \cite{DuuPar08} that any $k$-linear space of functions with poles only at infinity or at 
%a second point is spanned by rational functions of the form $h^k f_{i,j},$ for $k \in {\mathbb Z}$, $0 \leq i < q+1$,
%$0 \leq j < r$. \\

When we extend the plane curves to the GK-curves 
\[
\begin{array}{llll}
x^2+x = y^3, &y^4 - y = z^3 &~~ &\text{over ${\mathbb F}_{2^6}$}, \\
x^3+x = y^4, &y^9 - y = z^7 &~~ &\text{over ${\mathbb F}_{3^6}$}, \\
x^4+x = y^5, &y^{16}-y = z^{13} &~~ &\text{over ${\mathbb F}_{4^6}$}, 
\end{array}
\]
and choose as second point $(\alpha,\beta,\gamma)$ with $\gamma \neq 0$ then
we can use the same functions $f_{i,j}$ with the same vanishing orders (since $(\alpha,\beta,\gamma)$ is one of $q$ distinct
points lying above $(\beta,\gamma)$). The pole orders in the left table
on the other hand are all multiplied by $q$ (since the point at infinity is the unique point above the point at infinity on the plane curve, in a covering of degree $q$). In other words, the plane curve and the GK-curve share the same 
free basis for their two-point coordinate rings. The generating functions $h$ for the ring $k[h,h^{-1}]$ are different
in each case but are related via $h_{GK}^q + h_{GK} = h$, where we can choose $h_{GK} = Z$ and 
$h = w$ as in Section \ref{S:implicit} and Section \ref{S:explicit}. For the GK-curve, rational points $(\alpha,\beta,\gamma)$ 
with $\gamma \neq 0$ lie in a single orbit under the action of the automorphism group (\cite{GiuKor09}) and the tables do not depend
on the choice of the second point. \\

The tables contain all the information about two-point Weierstrass nongaps and in particular about one-point
Weierstrass nongaps. Of particular interest are the functions with leading monomial 
\[
\begin{array}{llclclclclcl}
(q=2) &Y  &(6,-1)  &Z &(8,-2)   &Z^2 &(16,-5) \\
(q=3) &Y  &(21,-1) &Z &(27,-3)  &Z^3 &(81,-10) &Z^5 &(135,-19) \\
(q=4) &Y  &(52,-1) &Z &(64,-4)  &Z^4 &(256,-17) &Z^7 &(448,-33) &Z^{10} &(640,-49) 
\end{array}
\]
For each function, the numbers $(a,b)$ in parentheses give the 
pole order $a$ at infinity and the vanishing order $-b$ at a point $(\alpha,\beta,\gamma)$ above $(\beta,\gamma)$, $\gamma \neq 0.$ 
After multiplication with a power of the function $h_{GK}$ we find functions
with poles only at the second point that vanish with maximal order at infinity. The numbers $(a',b')$ give the pole order $b'$ at the 
second point and the corresponding maximal order of vanishing $-a'$ at the point at infinty. 
\[
\begin{array}{llclclclclcl}
(q=2) &Y  &(-3,8)  &Z &(-1,7)   &Z^2 &(-2,13) \\
(q=3) &Y  &(-7,27) &Z &(-1,25)  &Z^3 &(-3,74) &Z^5 &(-5,121) \\
(q=4) &Y  &(-13,64) &Z &(-1,61)  &Z^4 &(-4,243) &Z^7 &(-7,422) &Z^{10} &(-10,601) 
\end{array}
\]
In this way we recover the numerical semigroups $\langle 7,8,9,13 \rangle$ $(q=2)$, $\langle 25, 27, 28, 74, 121 \rangle$ $(q=3)$, 
and $\langle 61,64,65,243,422,601 \rangle$ $(q=4)$. The cases $q=2,3$ were computed in \cite{FanGiu10}. We do not list the 
functions themselves, which are in general returned by the algorithm as rather long polynomials. \\

The pole orders in the left tables are minimal non-gaps within their residue class modulo $q^3+1$. With Proposition \ref{P:bijection}
this guarantees that the corresponding vanishing orders in the right table lie in the interval $[0,q^3+1).$ Proposition \ref{P:bijection}
applies to general curves. Both the plane curve $y^{q^2}-y = z^r$ and the generalized GK-curve $\Cn$ have the special property that
the canonical divisor is a multiple of the point at infinity. We indicate briefly how this can be used to explain in a different way 
that the vanishing orders in the right table lie in the interval $[0,q^3)$. The largest entry in the left table is the pole order of 
the monomial $y^q z^{r-1}$. For general $n$, this pole order is
\[
q \cdot r + (r-1) \cdot q^2 = q (q+1) r - q^2 = q^{n+1} + q - q^2,
\]
for the plane curve $y^{q^2}-y = z^r$, and $q^{n+2} + q^2 - q^3$ for the generalized GK-curve $\Cn$. In both cases the pole order
equals $2g-1+q^n+1$. This pole order is minimal within its residue class modulo $q^n+1$ and thus, for both the plane curve and the curve $\Cn$, 
$2g-1$ is a nongap for the point at infinity, and the canonical divisor is a multiple of the point at infinity.
To the pole order $2g-1+q^n+1$ corresponds the maximal vanishing order $q^n$. Using $K=(2g-2)\infty$ it can be shown that if $(a,b)$ 
is any pair of a pole order $a$ and a corresponding maximal vanishing order $b$ then $(a',b')$ is another such pair for $a+a' = 2g-1+q^n+1$
and $b+b' = q^n$. The claim corresponds to Lemma 8.2 in \cite{DuuPar08}. It follows from the characterization
\[
L(aP-bQ) \neq L((a-1)P-bQ) \quad \text{and} \quad  L(aP-(b+1)Q) = L((a-1)P-(b+1)Q,
\]
for pairs $(a,b)$, using the Riemann-Roch theorem together with the assumption $K=(2g-2)\infty.$ \\

A clear pattern emerges from the cases $q=2,3,4$. In terms of a general $q$ the observed patterns are the following.  
There exist functions with pole order $a$ at $\infty$ with maximal vanishing order $-b$ at $(\alpha,\beta,\gamma)$, for $\gamma \neq 0$,
for 
\[
\begin{cases}
&(a,b) = (q^3-q^2+q,-1), (q^3,-q), (q^3+1,-q^3-1), \quad \text{and for} \\
&(a,b) = (q^4+i(q^4-q^3),-q^2-1-iq^2), ~~ i=0,1,\ldots,q-2.
\end{cases}
\]
The corresponding set of pairs $(a',b')$ such that there exist functions with pole order $b'$ at $(\alpha,\beta,\gamma)$, for $\gamma \neq 0$,
with maximal order of vanishing $-a'$ at $\infty$ are
\[
\begin{cases}
&(a',b') = (-q^2+q-1,q^3), (-1,q^3+1-q), (-q^3-1,q^3+1), \quad \text{and for} \\
&(a',b') = (-q-i(q-1),q^4+q+i(q^4+q-q^3-1)-q^2-1-iq^2), ~~ i=0,1,\ldots,q-2,
\end{cases}
\]
In particular, the Weierstrass semigroup at $(\alpha,\beta,\gamma)$, for $\gamma \neq 0$, is
\[
\langle q^3-q+1, q^3, q^3+1, q^4-q^2+q-1+i(q^4-q^3-q^2+q-1) : i = 0,1,\ldots, q-2 \rangle.
\]
The patterns hold for other values of $q$ as well (we tested up to $q=9$, using Magma) but a proof of the general case 
seems to require a further analysis of the functions involved.

%(code available from
%the author's webpage) with the following running times in seconds: $q=2 (0.0), q=3 (0.0),
%q=4 (0.0), q=5 (0.3), q=7 (26.8), q=8 (35.8), q=9 (187.5).$ \\

%$(a,b)$, $(a',b')$, $i$ is total degree, "genus automatic". \\

%timing, symmetry. \\

%\[
%\begin {array}{l|rrrrrr} &1&Y&{Y}^{2}&{Y}^{3}&{Y}^{4}&{Y}^{5}\\ \hline 1&0&1&2&3&4&9\\ Z&5&6&7&8&13&14\\ {Z}^{2}&10&11&12&17&18&
%19\\ {Z}^{3}&15&16&21&22&23&24\\ {Z}^{4}&20&25&27&28&29&34\\ {Z}^{5}&26&30&32&33&38&39\\ {Z}^{6}&31&35&37&42&43&
%44\\ {Z}^{7}&36&40&45&47&48&49\\ {Z}^{8}&41&46&50&53&54&59\\ {Z}^{9}&51&52&55&58&63&64\\ {Z}^{10}&56&57&60&65&68
%&69\\ {Z}^{11}&61&62&67&70&73&74\\ {Z}^{12}&66&71&72&75&79&84\\ {Z}^{13}&76&77&78&80&85&89\\ {Z}^{14}&81&82&83&
%88&90&94\\ {Z}^{15}&86&87&92&93&95&99\\ {Z}^{16}&91&96&97&98&100&105\\ {Z}^{17}&101&102&103&104&109&110\\ {Z}^{
%18}&106&107&108&113&114&115\\ {Z}^{19}&111&112&117&118&119&120\\ {Z}^{20}&116&121&122&123&124&125\end {array}
%\]

\[
\begin{array}{ccc}

\begin {array}{l|rrr} &1&Y&{Y}^{2}\\ \hline 1&0&3&6
\\ Z&4&7&10\\ {Z}^{2}&8&11&14
\end {array}
&\qquad 
&\begin {array}{l|rrr} &1&Y&{Y}^{2}\\ \hline 1&0&1&3\\    Z&2&4&6 \\  {Z}^{2}&5&7&8\end {array} \\[1ex]
&\quad \\
\begin {array}{l|rrrr} &1&Y&{Y}^{2}&{Y}^{3}\\ \hline
1&0&7&14&21\\ Z&9&16&23&30\\ {Z}^{2}
&18&25&32&39\\ {Z}^{3}&27&34&41&48
\\ {Z}^{4}&36&43&50&57\\ {Z}^{5}&45&
52&59&66\\ {Z}^{6}&54&61&68&75\end {array}
&\qquad
&\begin {array}{l|rrrr} &1&Y&{Y}^{2}&{Y}^{3}\\ \hline 1&0&1&2&5\\Z&3&4&7&8\\{Z}^{2}&6&9&11&14\\{Z}^{3}&10&12&
15&17\\{Z}^{4}&13&16&18&21\\{Z}^{5}&19&20&23&24\\{Z}^{6}&22&25&26&27\end {array} \\[1ex]
&\quad \\
\begin {array}{l|rrrrr} 0&1&Y&{Y}^{2}&{Y}^{3}&{Y}^{4}
\\ \hline 1&0&13&26&39&52\\ Z&16&29&42&55&
68\\ {Z}^{2}&32&45&58&71&84\\ {Z}^{3
}&48&61&74&87&100\\ {Z}^{4}&64&77&90&103&116
\\ {Z}^{5}&80&93&106&119&132\\ {Z}^{
6}&96&109&122&135&148\\ {Z}^{7}&112&125&138&151&164
\\ {Z}^{8}&128&141&154&167&180\\ {Z}
^{9}&144&157&170&183&196\\ {Z}^{10}&160&173&186&199&
212\\ {Z}^{11}&176&189&202&215&228
\\ {Z}^{12}&192&205&218&231&244\end {array}
&\qquad
&\begin {array}{l|rrrrr} &1&Y&{Y}^{2}&{Y}^{3}&{Y}^{4}\\ \hline 1&0&1&2&3&7\\Z&4&5&6&10&11\\{Z}^{2}&8&9&13&14&15
\\{Z}^{3}&12&16&18&19&23\\{Z}^{4}&17&20&22&26&27\\{Z}^{5}&21&24&28&30&31\\{Z}^{6}&25&29&32&35&39
\\{Z}^{7}&33&34&36&40&43\\{Z}^{8}&37&38&42&44&47\\{Z}^{9}&41&45&46&48&52\\{Z}^{10}&49&50&51&55&56
\\{Z}^{11}&53&54&58&59&60\\{Z}^{12}&57&61&62&63&64\end {array} \\
\quad \\
\multicolumn{3}{c}{\text{Pole orders (left) and vanishing orders (right) for the curves}} \\
\multicolumn{3}{c}{\text{$y^4-y=z^3$ (top), $y^9-y=z^7$ (middle), $y^{16}-y=z^{13}$ (bottom).}} 
\end{array}
\]

\section{Conclusion}

We provided a new self-contained proof for the maximality of a generalized GK-curve.
Furthermore, we provided an efficient way to construct functions with prescribed poles
or vanishing orders at two given points $P$ and $Q$ on the curve, for $P$ the point at 
infinity and for $Q$ a rational point $(\alpha,\beta,\gamma)$ with $\gamma \neq 0.$ 
For the original GK-curve, we expect to be able to give generating functions in closed form 
for the ring of functions with poles only at $Q$. That would make it possible to settle
the structure of the Weierstrass semigroup at $Q$.

\nocite{AbdBezQuo09}
\nocite{BeeTut06}
\nocite{GiuKor09}
\nocite{KorTor01}
\nocite{DuuEnj02}
\nocite{Iha81}
\nocite{RucSti94}
\nocite{DuuPar08}
\nocite{GarGunSti07}
\nocite{FanGiu10}
\nocite{HirKorTor08}

\newpage
\def\lfhook#1{\setbox0=\hbox{#1}{\ooalign{\hidewidth
  \lower1.5ex\hbox{'}\hidewidth\crcr\unhbox0}}}

%\bibliography{dist-bounds}

\begin{thebibliography}{10}

\bibitem{AbdBezQuo09}
Miriam Abd{\'o}n, Juscelino Bezerra, and Luciane Quoos.
\newblock Further examples of maximal curves.
\newblock {\em J. Pure Appl. Algebra}, 213(6):1192--1196, 2009.

\bibitem{GarGunSti07}
Cem~G{\"u}neri Arnaldo~Garcia and Henning Stichtenoth.
\newblock A generalization of the Giulietti-Korchm{\'a}ros maximal curve.
\newblock {\em Adv. Geom.}, 10(3):427--434, 2010.

\bibitem{BeeTut06}
Peter Beelen and Nesrin Tuta{\c{s}}.
\newblock A generalization of the {W}eierstrass semigroup.
\newblock {\em J. Pure Appl. Algebra}, 207(2):243--260, 2006.

\bibitem{DuuEnj02}
Iwan Duursma and Jean-Yves Enjalbert.
\newblock Bounds for completely decomposable {J}acobians.
\newblock In {\em Finite fields with applications to coding theory,
  cryptography and related areas ({O}axaca, 2001)}, pages 86--93. Springer,
  Berlin, 2002.

\bibitem{DuuPar08}
Iwan Duursma and Seungkook Park.
\newblock Coset bounds for algebraic geometric codes.
\newblock {\em Extended version}, arXiv:0810.2789, 2008.

\bibitem{FanGiu10}
Stefania Fanali and Massimo Giulietti.
\newblock One-point {AG} codes on the {GK} maximal curves.
\newblock {\em IEEE Trans. Inform. Theory}, 56(1):202--210, 2010.

\bibitem{GiuKor09}
Massimo Giulietti and G{\'a}bor Korchm{\'a}ros.
\newblock A new family of maximal curves over a finite field.
\newblock {\em Math. Ann.}, 343(1):229--245, 2009.

\bibitem{HirKorTor08}
J.~W.~P. Hirschfeld, G.~Korchm{\'a}ros, and F.~Torres.
\newblock {\em Algebraic curves over a finite field}.
\newblock Princeton Series in Applied Mathematics. Princeton University Press,
  Princeton, NJ, 2008.

\bibitem{Iha81}
Yasutaka Ihara.
\newblock Some remarks on the number of rational points of algebraic curves
  over finite fields.
\newblock {\em J. Fac. Sci. Univ. Tokyo Sect. IA Math.}, 28(3):721--724 (1982),
  1981.

\bibitem{Kim94}
Seon~Jeong Kim.
\newblock On the index of the {W}eierstrass semigroup of a pair of points on a
  curve.
\newblock {\em Arch. Math. (Basel)}, 62(1):73--82, 1994.

\bibitem{KorTor01}
G{\'a}bor Korchm{\'a}ros and Fernando Torres.
\newblock Embedding of a maximal curve in a {H}ermitian variety.
\newblock {\em Compositio Math.}, 128(1):95--113, 2001.

\bibitem{Mat01}
Gretchen~L. Matthews.
\newblock Weierstrass pairs and minimum distance of {G}oppa codes.
\newblock {\em Des. Codes Cryptogr.}, 22(2):107--121, 2001.

\bibitem{RucSti94}
Hans-Georg R{\"u}ck and Henning Stichtenoth.
\newblock A characterization of {H}ermitian function fields over finite fields.
\newblock {\em J. Reine Angew. Math.}, 457:185--188, 1994.

\end{thebibliography}
%\bibliographystyle{plain}

\end{document}